\documentclass[11pt]{amsart}
\usepackage{amssymb}
\usepackage{latexsym}
\usepackage{amscd} 
\usepackage{amsmath}
\usepackage{epsfig}
\usepackage{mathrsfs}
\usepackage[margin=3.25cm]{geometry}
\usepackage{ifpdf}
\usepackage[textures,backref=false,bookmarks=false]{hyperref}
\usepackage{xr}
\usepackage{bm}

\usepackage{amsthm}
\newtheorem{theorem}{Theorem}[section]
\newtheorem{introtheorem}{Theorem}

\newtheorem{corollary}[theorem]{Corollary}

\newtheorem{lemma}[theorem]{Lemma}
\newtheorem{definition}[theorem]{Definition}
\newtheorem{proposition}[theorem]{Proposition}

\def\C{\mathbb{C}}

\def\R{\mathbb{R}}
\def\O{\mathbb{O}}
\def\P{\mathbb{P}}
\def\Q{\mathbb{Q}}

\def\N{\mathbb{N}}
\def\L{\mathbb{L}}
\def\Z{\mathbb{Z}}

\def\CP2{{\mathbb{CP}^2}}

\def\ratdc{{\operatorname{Rat}^d_\C}}
\def\ratd{{\operatorname{Rat}^d}}
\def\rat{{\operatorname{Rat}}}

\def\cA{{\mathcal{A}}}
\def\cB{{\mathcal{B}}}
\def\cC{{\mathcal{C}}}

\def\cH{{\mathcal{H}}}

\def\cO{{\mathcal{O}}}

\def\z{\zeta}
\def\vphi{\varphi}

\def\laurent{\C((t))}
\def\laurentm{\C((t^{1/m}))}
\def\puiseux{\C \langle\langle t \rangle\rangle }
\def\ponel{\P^1_\L}
\def\ponec{\P^1_\C}
\def\pone{\P^1}

\def\juliavphi{J(\vphi)}

\def\pberl{\P^{1,an}_{\L}}

\def\seqMn{{\{M_n\}}}
\def\seqMjn{{\{M_{j,n}\}}}
\def\seqMnprime{{\{M'_n\}}}
\def\seqLn{{\{L_n\}}}
\def\seqLnprime{{\{L'_n\}}}

\def\seqfn{{\{f_n\}}}

\def\moebiusc{{\operatorname{PSL}(2,\C)}}

\def\vecz{{\vec{z}}}
\def\vecw{{\vec{w}}}

\newcommand{\mish}[1]{\ensuremath{\{ #1_t} \} }
\newcommand{\seq}[1]{\ensuremath{\{ {#1}_n} \} }
\newcommand{\mishl}[1]{\ensuremath{ \boldsymbol{#1}}  }

\newcommand{\tang}[1]{\ensuremath{T_{#1} {\pberl}}}

\newcommand{\cball}[2]{\ensuremath{B_{\le #1} (#2)}}
\newcommand{\oball}[2]{\ensuremath{B_{< #1} (#2)}}

\begin{document}

\title{Rescaling limits of complex rational maps}
\author{Jan Kiwi}
\address{Facultad de Matem\'aticas,
Pontificia Universidad Cat\'olica de
Chile.}
\thanks{Supported by ``Proyecto Fondecyt \#1110448''}
\email{jkiwi@uc.cl}
\date{\today}
\begin{abstract}
  We discuss rescaling limits for sequences of complex rational maps in one variable  which approach infinity in parameter space.
  It is shown that  any given sequence of maps of degree $d \ge 2$ has at most $2d-2$ dynamically distinct rescaling limits which are not postcritically finite.
For quadratic rational maps, a complete description of the possible rescaling limits is given. 
These results are obtained employing tools from non-Archimedean dynamics.
\end{abstract}

\maketitle


\section{Introduction}

The parameter space  $\ratdc$ of  degree $d$ 
rational maps in one complex  variable  
is naturally identified with a Zariski open subset of $\P^{2d+1}_\C$.
That is, the complement of $\ratdc$ in $\P^{2d+1}_\C$ is a codimension one algebraic variety which coincides with $\partial \ratdc$.
When rational maps are viewed as dynamical systems acting on the Riemann sphere $\ponec$, it is of interest
to understand what happens to the dynamics as they approach $\partial \ratdc$. 
The aim of this paper is to study ``rescaling limits'' which are non-trivial dynamical systems that arise in this context.
We establish, in a sense to be precised later,
 that most rescaling limits are postcritically finite rational maps.
Moreover, we give a complete characterization of the rescaling limits which occur for quadratic rational maps\footnote{The author learned of the question: ``How many dynamically distinct limits can a sequence of quadratic rational maps have?'' from a conversation with Laura De Marco and Adam L. Epstein in 2005.}.
These results arise after applying tools from non-Archimedean dynamics.

Maybe not in this language, ``rescaling limits'' already appear in the literature.
Specially for $d=2$, since quadratic rational maps have been studied in greater detail.
More precisely, rescaling limits appear in Stimson's Ph.D. Thesis~\cite{Stimson}  to describe the
asymptotic behavior of certain algebraic curves in quadratic moduli space,
in Epstein's proof that hyperbolic components of certain type are precompact in quadratic moduli space~\cite{EpsteinBounded} and, in
De Marco's description of a  compactification of quadratic moduli space where the iteration map extends continuously~\cite{DeMarcoQuadratic}. 
In Section~\ref{sec:examples}, we
also show how rescaling limits naturally arise in other parameter spaces: cubic polynomials, Latt\`es maps, and McMullen's Cantor cross circle Julia sets.

Although our results apply to sequences $\{ f_n \}$ of rational maps approaching $\partial \ratdc$,
non-Archimedean dynamics emerges  once degenerate holomorphic 
families $\{ f_t \} \subset \P^{2d+1}_\C$ are taken into account.
Here $\{ f_t \}$ is parametrized by a neighborhood of the origin in $\C$ and
$f_t \in \partial \ratdc$ if and only if $t=0$. Each such family may be regarded as a rational map $\mishl{f}$ with coefficients
in the field of formal Laurent series $\laurent$. However, we prefer to regard the coefficients as elements of
a complete and algebraically closed field containing $\laurent$ which we denote by $\L$. Then, following ideas nowadays standard, the action of
$\mishl{f} : \ponel \to \ponel$  becomes easier to understand when extended to the Berkovich projective line $\pberl$ over $\L$. 
We will explain how the ``rescalings'' of  $\{ f_t \}$ correspond to some of the periodic points in the Julia set of $\mishl{f} : \pberl \to \pberl$.
Exploiting some basic properties of dynamics on the Berkovich projective line we will obtain our results regarding the rescaling limits of holomorphic families 
which, after some work, are translated to results regarding  rescaling limits of sequences of rational maps.

Now we procede to give precise statements of  our results, first for holomorphic families and then for sequences of rational maps.

\subsection*{Rational maps.}
\label{sec:statements}
As mentioned above, the space of complex rational maps  of degree $d \ge 1$, denoted $\ratdc$,
is identified to a Zariski open subset of $\P_{\C}^{2d+1}$ via the inclusion:
$$
\begin{array}{ccc}
\ratdc & \hookrightarrow & \P_{\C}^{2d+1} \\
 \dfrac{a_0z^d+\cdots+a_d}{b_0z^d+ \cdots +b_d}& \mapsto & [a_0:\dots:a_d:b_0:\dots:b_d].
\end{array}$$
Hence $\ratdc$ is identified with the complement of the hypersurface in $\P_{\C}^{2d+1}$ obtained as the vanishing locus
of the resultant of the polynomials ${a_0z^d+\cdots+a_d}$ and ${b_0z^d+ \cdots +b_d}$.

\subsection*{Holomorphic families}
Given a neighborhood $U$ of the origin in $\C$, 
a collection $\mish{f}_{t \in U} \subset \P_{\C}^{2d+1} $ is 
a {\sf one--dimensional holomorphic family of degree $d \geq 1$} 
if 
$$
\begin{array}{ccl}
  U & \to & \P_{\C}^{2d+1}\\
          t      & \mapsto & f_t
\end{array}
$$
is a holomorphic map such that $f_t \in \rat_{\C}^d $  for all $t \neq 0$.
We say that  $\mish{f}_{t \in U}$ is a {\sf degenerate} holomorphic family if 
 $f_0 \notin \rat_{\C}^d $.
A holomorphic family $\mish{M}_{t \in U}$ of degree $1$ will be called, following Shishikura, a {\sf moving frame}.

Holomorphic families of rational maps appear in the literature in a more general setting which allows parametrizations by  complex manifold of any dimension and any 
topology. In this paper we will only consider the very special kind of holomorphic families parametrized by neighborhoods of the origin in $\C$ and proceed to simply call
them ``holomorphic families''.
Also, we drop the subscript ``$t \in U$''
 from $\mish{f}_{t \in U}$, since
 all of our discussion about holomorphic families only depends
on the corresponding germs at $t=0$ and not on the domain $U$.
However, for the sake of simplicity of exposition we 
prefer to work with holomorphic families instead of with germs.

\begin{definition}
  \label{def:5}
  Let $\mish{f}$ be a holomorphic family of degree $\ge 2$. We say that a moving frame $\mish{M}$ is a {\sf rescaling for $\mish{f}$}, if there exists  an integer $q \ge  1$, a degree $d' \geq 2$ rational map $g: \ponec \to \ponec$ and a finite subset $S$ of $\ponec$ such that
  \begin{equation}
    \label{eq:4}
    M_t^{-1} \circ f_t^q \circ M_t (z) \to g (z), \, \mbox{ as } t \to 0,
  \end{equation}
uniformly on compact subsets of $\ponec \setminus S$.
We say that $g$ is a {\sf rescaling limit for $\mish{f}$ in  $\ponec \setminus S$}.
The minimal $q \ge 1$ such that the above holds is called the {\sf period of the rescaling $\mish{M}$}.
\end{definition}

We will show  that any $q$ for which~\eqref{eq:4} holds for some $g$ of degree at least $2$ is a multiple of the period of $\mish{M}$ (see Corollary~\ref{cor:3}).

\begin{definition}
  \label{def:4}
  We say that two moving frames $\mish{M}$ and $\mish{L}$ are {\sf equivalent} if there exists $M \in \rat^1_\C$ such that  
$M^{-1}_t \circ L_t \to M$
as $t \to 0$.
The equivalence class of $\mish{M}$ will be denoted by $[\mish{M}]$ and
the set formed by the  equivalence classes of moving frames by $\cB$.
\end{definition}

It is not difficult to check that the above relation among moving frames is an equivalence relation.
Also, if the moving frame $\mish{M}$ is a rescaling for $\mish{f}$ with rescaling limit $g$ and $\mish{L}$ is equivalent to $\mish{M}$, then
$\mish{L}$ is a rescaling for  $\mish{f}$ with rescaling limit M\"obius conjugate to $g$.

In Section~\ref{sec:equiv-moving-fram} we will translate results from non-Archimedean dynamics to show that any holomorphic family $\mish{f}$ acts 
on the set formed by the equivalence classes of moving frames $\cB$. 
More precisely, given an equivalence class of moving frames $[\mish{M}]$
there exists a unique class of  moving frames $[\mish{L}]$
such that: 
\begin{equation*}
  L_t^{-1} \circ f_t \circ M_t (z) \to g(z), \, \mbox{ as } t \to 0,
\end{equation*}
for some non-constant rational map $g: \ponec \to \ponec$ and
for all $z$ outside some finite set. 
It follows that $\mish{M}$ is a rescaling if and only if $\mish{L}$ is a rescaling.
Thus it is reasonable to say that two rescalings $\mish{M}$ and $\mish{L}$ 
for $\mish{f}$ are {\sf dynamically independent} 
only when the corresponding classes 
$[\mish{M}]$ and $[\mish{L}]$ lie 
in distinct $\mish{f}$-orbits.

\begin{introtheorem}
  \label{ithr:3}
 Let $\mish{f}$ be a holomorphic family of degree $d \ge 2$ rational maps. 
Then there are at most $2d -2$ pairwise dynamically independent 
rescalings for $\mish{f}$ such that the corresponding rescaling limits are not postcritically finite.

 Moreover, if $d=2$, then there are at most two dynamically independent 
rescalings. Furthermore, in the case that a rescaling of period at least $2$ exists, then exactly one of the following holds:
 \begin{enumerate}
 \item  $\mish{f}$ has exactly two dynamically independent rescalings, 
of periods $q' > q > 1$. The period $q$ rescaling limit is  a quadratic rational map with a multiple fixed point and a prefixed critical point. The period $q'$ rescaling limit is a quadratic polynomial, modulo conjugacy.
\item  $\mish{f}$ has a rescaling whose corresponding limit is 
a quadratic rational map with a multiple fixed point and every other rescaling
is dynamically dependent to it.
 \end{enumerate}
\end{introtheorem}

A degree $2$ holomorphic family with a period $1$ rescaling limit is, modulo conjugacy by a moving frame, a non-degenerate holomorphic family and lacks of non-trivial rescalings (e.g., see Proposition~\ref{pro:2}). 

\subsection*{Sequential rescalings}
\label{sec:sequ-resc}
In many situations, the study of parameter spaces  leads us to  consider sequences of rational maps. Thus, it is of interest to explore the sequential analogues of the results previously described for holomorphic families. 

We say that a sequence $\seq{M}$ of M\"oebius transformations is a {\sf sequence of frames}. That is, $M_n \in \moebiusc \equiv \rat^1_\C$.

\begin{definition}
  \label{def:1}
  Consider a sequence $\{ f_n \}$ in $\ratdc$. We say that a sequence of frames $\{ M_n \}$ is a {\sf rescaling for $\seqfn$} 
if there exists an integer $q \ge  1$, a degree $d' \geq 2$ rational map $g: \ponec \to \ponec$ and a finite subset $S$ of $\ponec$ such that
\begin{equation}
  \label{eq:5}
  M_n^{-1} \circ f_n^q \circ M_n (z) \to g (z), \, \mbox{ as } n \to \infty,
\end{equation}
uniformly on compact subsets of $\ponec \setminus S$.
We say that $g$ is a {\sf rescaling limit for $\seqfn$}.
The minimal $q \ge 1$ such that the above holds is called the {\sf period of the rescaling $\seqMn$}.
\end{definition}

We will show that any $q$ for which~\eqref{eq:5} holds for some $g$ of degree at least $2$ is a multiple of the period of $\seq{M}$ (see Section~\ref{sec:period-rescaling}).

\begin{definition}
  \label{def:2}
  Two sequences of frames $\seqMn$ and $\seqLn$  are said to be {\sf independent} if  $M_n^{-1} \circ L_n \to \infty$ in $\moebiusc$. That is, for every compact set $K$ in
$\moebiusc$ there exists $n_0$ such $M_n^{-1} \circ L_n \notin K$ for all $n \ge n_0$.

Two rescaling sequences $\seqMn$ and $\seqLn$  are said to be {\sf equivalent} if  $M_n^{-1} \circ L_n \to M \in \moebiusc$.
\end{definition}

Note that equivalent sequences of frames are dependent (i.e. not independent).
However, if $\seqMn$ and $\seqLn$ are two dependent sequences of frames, then we may only conclude that there exist subsequences $\{ M_{n_k} \}$ and $\{ L_{n_k} \}$ which are equivalent.
Thus, the sequential analogue of dynamical independence  will require
passing to subsequences.

\begin{definition}
  \label{def:3}
  Given a sequence $\{ f_n \}$ in $\ratdc$ and rescalings $\seqMn$ and $\seqLn$ for $\seqfn$ of period dividing $q$. 
We say that  {\sf $\seqMn$ and $\seqLn$ are dynamically dependent} if, for some subsequences $\{ M_{n_k} \}$ and $\{ L_{n_k} \}$,  there exist $1 \leq m \leq q$, finite subsets $S_1$, $S_2$
of $\ponec$ and rational maps $g_1, g_2 : \ponec \to \ponec$ of degree at least $1$ 
 such that 
$$L^{-1}_{n_k} \circ f_{n_k}^m \circ M_{n_k} (z) \to g_1 (z)$$
uniformly on compact subsets of $\ponec \setminus S_1$ and
$$M^{-1}_{n_k} \circ f_{n_k}^{q-m} \circ L_{n_k} (z) \to g_2 (z)$$
uniformly on compact subsets of $\ponec \setminus S_2$.
\end{definition}

With the above definitions we will establish the sequential version of
Theorem~\ref{ithr:3}:

\begin{introtheorem}
\label{ithr:2}
  Let $d \geq 2$ and consider $\seqfn \subset \ratdc$. Then there exists at most 
$2d-2$ pairwise dynamically independent rescalings whose rescaling limits are not postcritically finite.

Moreover, if $d=2$, then there are at most two dynamically independent rescaling limits of period at least $2$.
Furthermore, in the case that a rescaling of period at least $2$ exists, then exactly one of the following holds:
 \begin{enumerate}
 \item  $\seq{f}$ has exactly two dynamically independent rescalings, 
of periods $q' > q > 1$. The period $q$ rescaling limit is  a quadratic rational map with a multiple fixed point and a prefixed critical point. The period $q'$ rescaling limit is a quadratic polynomial, modulo conjugacy.
\item  $\seq{f}$ has a rescaling whose corresponding limit is 
a quadratic rational map with a multiple fixed point and every other rescaling
is dynamically dependent to it.
 \end{enumerate}
\end{introtheorem}

\subsection*{Outline}
This paper is organized as follows. In Section~\ref{sec:examples} we give several examples of rescaling limits for holomorphic families as well as for sequences of rational maps. In Section~\ref{sec:hol2ber} we discuss how each holomorphic family determines a rational map with coefficients in a non-Archimedean field $\L$. We show
how rescaling limits of a given holomorphic family correspond to certain periodic points of the corresponding non-Archimedean dynamical system. 
These periodic points lie in the Berkovich projective line over $\L$.
In Section~\ref{sec:ber2res} we employ the tree structure of the Berkovich projective line to prove Theorem~\ref{ithr:3}.
Finally, in Section~\ref{sec:sequential}, we reduce the proof of Theorem~\ref{ithr:2} to   Theorem~\ref{ithr:3} after showing 
that sequential rescalings are realized by rescalings of holomorphic families.

\subsection*{Acknowledgements} 
I would like to express my gratitude to Laura De Marco and Adam L. Epstein for several conversations which motivated my interest for the questions addressed here.

\section{Examples}
\label{sec:examples}
The aim of this section is to illustrate rescaling limits
through examples. All the examples below are motivated
by holomorphic families 
 that have been already considered in the literature.

\subsection{Cubic polynomials}
Consider the family of cubic polynomials with one period $3$ critical point $w=0$ which, following Milnor~\cite{MilnorPeriodicCubic}, can be parametrized by
$$F_c (w) = \alpha (c) w^3 + \beta (c) w^2 +1$$
where
$$\alpha (c) = -\dfrac{c^3+2c^2+c+1}{c(c+1)^2}, \quad \beta (c)= c-\alpha(c).$$
The polynomial $F_c$ 
is well defined and cubic for all $c \in \C$ such that
 $\alpha (c) \neq 0, \infty$. 

Polynomials of this form  lead to six different degenerate
holomorphic families of degree $3$ that correspond
to the six points $c \in \C \cup \{ \infty \} \equiv \ponec$ for which $\alpha (c) = 0$ or $\infty$. Three of these degenerate holomorphic families will have a rescaling of period $1$ while the other three will have a rescaling of period $3$. 

If $\alpha(c)$ vanishes at $c=c_0$,  then the holomorphic family $\{ F_{t+c_0} \}$ is a degenerate holomorphic family of degree $3$, 
defined in a neighborhood of $t=0$. 
The trivial  moving frame  $\{M_t: z \mapsto z \}$ is a rescaling of period $1$ for
$\{ F_{t+c_0} \}$ with rescaling limit $c_0 w^2 +1$. Here
$F_{t+c_0}$ converges to  $c_0 w^2 +1$
uniformly on compact subsets of $\C$. Note that $c_0w^2 + 1$ is affine conjugate to $z^2 + c_0$. The roots of $\alpha(c)$  are exactly the
three parameters $c$ for which the critical point $z=0$
has period $3$ under iterations of  the standard quadratic family $z \mapsto z^2+c$ (i.e. the parameters of the Julia sets known as the rabbit, anti-rabbit, and airplane).

Now, for parameters $c$ close to $c=0$, we have that the moving 
frame $\{ M_t: z \mapsto t^2 z \}$ is a period $3$ rescaling for 
$\{ F_t \}$ with limit $g(z) = z^2$. 
Close to $c=-1$, the moving frame
$\{ z \mapsto t^5 z \}$ is a rescaling
of period $3$ for the family $\{ F_{-1 + t} \}$ with rescaling limit 
$g(z) = 2 z^2$. 
Finally, close to $c = \infty$, the moving frame
$\{ z \mapsto t^4 z \}$ is a rescaling
of period $3$ for the family $\{ F_{1/t} \}$ with rescaling limit 
$g(z) = - 2 z^2$. 

\medskip
This example illustrates a general phenomena which occurs when one considers degenerate holomorphic families that parametrize a neighborhood of infinity of an escape region of a periodic curve of cubic polynomials. The correct choice of rescaling $\{ M_t \}$ above is based 
on~\cite[Section~5]{AKMCubic}.

\subsection{Quadratic rational maps}
Given $a \in \C$ and $t \in \C \setminus \{ 0 \}$, let
    $$f_{a,t} (z) = {t} - \dfrac{1+t^2}{z}+ \dfrac{t}{z^2} -at^5.$$ 
For each $a \in \C$  we have that
 $\{f_{a,t} \}$ is a degenerate holomorphic family of degree $2$.

Setting $a=0$,  
 the critical point $z=0$ becomes periodic of period
$3$ under iterations of $f_{0,t}$. 
Since we are only interested on small values
of $t$, for each $a$, the holomorphic family $\{f_{a,t} \}$ may be regarded as a perturbation of $\{ f_{0,t} \}$. 

As anticipated by Stimson~\cite{Stimson} and
Epstein~\cite{EpsteinBounded} there exists a rescaling limit of these
families which is a quadratic rational map with a multiple fixed
point.  More precisely, given any $a \in \C$, the moving frame $\{ z
\mapsto tz\}$ is a rescaling of period $2$ for $\{f_{a,t} \}$ with
rescaling limit $$g (z) = \dfrac{z^2+z-1}{z-1}.$$ The convergence of
$M^{-1}_t \circ f^2_t \circ M_t$ is uniform on compact subsets of $\C
\setminus \{1 \}$. Note that $g$ has a multiple fixed point at
$\infty$ and the critical point $z=0$ of $g$ maps into this fixed
point in two iterations.  

Also, the moving frame $\{ L_t: t \mapsto t^3
z \}$ is another rescaling for $\{f_{a,t} \}$, but now of period $3$ and
rescaling limit $z \mapsto z^2 +a$, where the convergence of $L^{-1}_t
\circ f^3_t \circ L_t $ is uniform on compact subsets of $\C$.

\subsection{Latt\`es maps}
Consider the degree $4$ flexible Latt\`es family (e.g. see~\cite[Problem~7-g]{MilnorComplexBook}) given by
$$f_t(w) = \dfrac{(w^2 -t)^2}{4w(w-1)(w-t)}.$$
It is a degenerate holomorphic family.
The trivial moving frame $\{ M_t: z \mapsto z \}$
leads to a period $1$ rescaling limit, 
$$w \mapsto \dfrac{w^2}{4(w-1)},$$
which after a change of coordinates
is the quadratic Chebyshev polynomial $z \mapsto z^2 -2$ (e.g. see~\cite[Problem~7-c]{MilnorComplexBook}).
For more about Latt\`es maps see~\cite{MilnorLattes} and the references therein.

There are a wealth of rescaling sequences associated to this family $\mish{f}$.
In fact, as we shall see below, one for each 
periodic point of the tent map:
$$\operatorname{Tent} (\alpha) = \begin{cases}
2 \alpha & \mbox{ if } \alpha \le 1/2, \\
-2 \alpha + 2& \mbox{ if } \alpha > 1/2.
\end{cases}$$

We restrict to $t \in ]0, +\infty[$ so that $t^\alpha \in ]0,
+\infty[$ is well defined for all $\alpha \in \R$. As $t \to 0$, 
$$ \dfrac{f_t (t^\alpha \cdot z)}{t^{\operatorname{Tent}(\alpha)}} \to 
\left\{ \begin{array}{ll} - \dfrac{z^2}4 & \text{if  } 0< \alpha < 1/2, \\
& \\
-\dfrac{z^{-2}}4 & \text{if  } 1/2 < \alpha < 1,
\end{array} \right.
$$
uniformly on compact subsets of $\C^\times = \C \setminus \{ 0 \}$.

Now we consider a sequence $\{ t_n \}$ such that
$t_n \searrow 0$ and let $\{ f_{t_n} \}$ be the corresponding
 sequence  of degree $4$ rational maps.
For  $\alpha$ periodic, say of period $q$,
 under $\operatorname{Tent}$, we have that
$\{ M_n (z) = t_n^\alpha z\}$ is a sequential rescaling for
$\{ f_{t_n} \}$ with limit affine conjugate to $z \mapsto z^{\pm 2^q}$.
For example, we may consider $\alpha=2/5$ which has period $2$ 
and let $s_n = t^{2/5}_n$. It follows that, as $n \to \infty$,
$$\dfrac{1}{s_n} \cdot f^2_{t_n} (s_n \cdot z) \to - 4 {z^{-4}}.$$

It is not difficult to check that rescalings associated to distinct periodic orbits of $\operatorname{Tent}$ are dynamically independent rescalings for  $\{ f_{t_n} \}$ (see Definition~\ref{def:3}). Hence, there are infinitely many pairwise ``distinct'' rescalings for $\{ f_{t_n} \}$. However, all rescaling limits are a special type of postcritically finite maps, namely  monomials.

\subsection{Cantor cross circle}
Now we discuss an example which, from the viewpoint of rescaling limits, 
is fairly similar to the flexible Latt\`es maps discussed above.
It is based on an example introduced and discussed by McMullen in~\cite[Section~7]{McM88}
where he shows that for sufficiently small values of $t$ the degree $5$
rational map
$$f_t(z) = z^3 + \dfrac{t}{z^2}$$
has Julia set homeomorphic to a circle cross a Cantor set.

The analogue of the tent map now is given by:
$$\beta(\alpha) = 
\begin{cases}
  3 \alpha & \mbox{ if } \alpha \leq 1/5, \\
  -2 \alpha +1  & \mbox{ if } \alpha > 1/5.
\end{cases}
$$
We also restrict to $t \in ]0, +\infty[$ and observe that, as $t \to 0$, uniformly on compact subsets of $\C^\times$ we have:
$$ \dfrac{f_t (t^\alpha \cdot z)}{t^{\beta(\alpha)}} \to 
\left\{ \begin{array}{ll} {z^3} & \text{if  } \alpha < 1/5, \\
& \\
z^{-2} & \text{if  } 1/5 < \alpha.
\end{array} \right.
$$

Periodic points of $\beta$ lie in  $[0,1/6] \cup [1/4,1/2]$ and are in one to one correspondence with those of the one-sided full shift on two symbols. 
Choose $t_n \searrow 0$.
For each periodic point $\alpha$, say of period $q$ under iterations of $\beta$,  we have that $\{M_{t_n} : z \to t_n^\alpha z \}$ is a period $q$ rescaling for $\{ f_{t_n} \}$
with rescaling limit $z \mapsto z^{{3^m} \cdot {(-2)^{q-m}}}$, where $m$ is the number of elements of the orbit of $\alpha$ under $\beta$ contained in $[0,1/6]$.
Periodic points $\alpha$ which lie in distinct $\beta$-orbits lead to dynamically independent rescalings for $\{ f_{t_n} \}$.

\section{From holomorphic families to Berkovich dynamics}
\label{sec:hol2ber}
To regard a holomorphic family $\mish{f}$ as a single dynamical system we consider the {\em field of formal Puiseux series} $\puiseux$.
This field $\puiseux$ is an algebraic closure
of the field of formal Laurent series $\laurent$. More precisely, $\puiseux$ is the injective limit of $\{ \laurentm \}_{m \in \N}$ with the obvious inclusions. 
The order of vanishing at $t=0$ induces a non-Archimedean absolute value 
$| \cdot |$ on $\puiseux$. That is,
given an element $z \in \puiseux$ we may consider $m \in \N$,  $j_0 \in \Z$ and $c_j \in \C$ such that $z = \sum_{j \ge j_0} c_j t^{j/m}$, then, provided that $z \neq 0$, we have
$$|z| = \exp \left(- \min \left\{\dfrac{j}{m} \mid c_j \neq 0 \right\} \right).$$

Although $\puiseux$ is algebraically closed (e.g. see~\cite[Corollary~1.5.11]{LibroCasas}) it is not complete with respect to
$| \cdot |$. For us, it is more convenient to work with the field $\L$ obtained as the completion of $\puiseux$. It follows 
that $\L$ enjoys being both complete and algebraically closed~\cite[Chapter~5, J]{RibenboimBook}.

\begin{definition}
  \label{def:7}
  Consider a  degree $d \ge 1$ holomorphic family $\mish{f}$. We may write
$$f_t (z) = \dfrac{a_0(t) z^d + \cdots + a_d (t)}{b_0 (t) z^d + \cdots + b_d(t)}$$
where  $a_j (t), b_j(t)$ are holomorphic functions whose domains contain a neighborhood of the origin, for all $j=0, \dots,d$.
Let $\boldsymbol{a_j}, \boldsymbol{b_j} \in \L$ be the Taylor series at  $t=0$ of the holomorphic functions $a_j (t), b_j(t)$.
Then the degree $d$ rational map $\boldsymbol{f}: \ponel \to \ponel$ given by
\begin{equation}
  \label{eq:6}
  \boldsymbol{f} (z) = \dfrac{\boldsymbol{a_0}  z^d + \cdots + \boldsymbol{a_d} }{\boldsymbol{b_0}  z^d + \cdots + \boldsymbol{b_d}}
\end{equation}
is called {\sf the rational map associated to $\mish{f}$}.
\end{definition}

Since $f_t$ has degree $d$ for all $t \neq 0$ the degree of the associated map $\boldsymbol{f}:  \ponel \to \ponel$ is $d$.

We will denote the rational map in $\L (z)$ associated to a holomorphic family by the corresponding boldface symbol.
That is, the rational map associated to $\mish{g}$, $\mish{M}$, $\mish{L}$ will be respectively denoted by $\mishl{g}$, $\mishl{M}$, $\mishl{L}$.

We will study the associated map $\boldsymbol{f}$ with tools from iteration of rational maps over non-Archimedean fields 
(e.g. see~\cite{BakerRumelyBook,JonssonDynBer}).
Since fields such as $\L$ are totally disconnected and not locally compact it is convenient, and nowadays standard, to extend the action of $\boldsymbol{f}$
to the Berkovich projective line $\pberl$ over $\L$. 

\subsection{The field $\L$}
Before working on Berkovich space, 
let us discuss the basic properties
of the field $\L$. 

As the elements of $\puiseux$, the elements of $\L$ may also be 
represented by series in $t$ but now 
of the form
$$z= \sum_{j \ge 0} a_j t^{\lambda_j},$$
where $a_j \in \C$, $\lambda_j \in \Q$ and, if $a_j$ does not vanish
for sufficiently large $j$, then $\lambda_j \to \infty$ as $j \to
\infty$.  The absolute value is given by $|z| = \exp (-\min \{ \lambda_j \mid a_j \neq 0 \})$
provided $z \neq 0$.  Hence, $|z|$ vanishes or takes values in the
multiplicative group $\exp(\Q)$ called the {\sf value group
  $|\L^\times|$ of $\L$}.

Given $z_0 \in \L$ and $r >0$ we say that  $\cball{r}{z_0} = \{ z \in \L \mid |z-z_0 | \leq r \}$ is a  {\em closed ball} and
$\oball{r}{z_0} =  \{ z \in \L \mid |z-z_0| < r \}$ is an {\em open ball}. However, all of these balls  are open and closed sets in the topology of $\L$.
When $r \notin \exp(\Q)$ note that $\cball{r}{z_0} =\oball{r}{z_0}$. 

The {\sf valuation ring} is the  local ring $\mathfrak{O}_\L = \{ | \cdot | \le 1 \}$ 
with   maximal ideal $\mathfrak{M}_\L = \{ | \cdot | < 1 \}$.
 The {\sf residue field} $\widetilde{\L} = \mathfrak{O}_\L/\mathfrak{M}_\L$ is canonically identified with $\C$
via $\C \ni c \mapsto c + \mathfrak{M}_\L \in \mathfrak{O}_\L/\mathfrak{M}_\L$. For more about valuations and local fields see, for example,~\cite{LibroCassels} and~\cite{RibenboimBook}.

When convenient we fix a coordinate and identify $\ponel$ with $\L \cup \{ \infty \}$ via $[z:1] \mapsto z$ and $[1:0] \mapsto \infty$. 
Similarly, we identify $\ponec$ with $\C \cup \{ \infty \}$. 

We may {\sf reduce} points of $ \ponel$ to  $\ponec$ via the map $\rho: \ponel \to \ponec$ which  is obtained after extending the quotient map $\mathfrak{O}_\L \to \mathfrak{O}_\L/\mathfrak{M}_\L \equiv \C \subset \ponec$ to $\ponel$ by declaring that  $\ponel \setminus \mathfrak{O}_\L$ maps onto
$\infty \in \ponec$. 

Given any rational map 
$\varphi: \ponel \to \ponel$, 
there exists a rational map $\widetilde{\varphi}: \ponec \to \ponec$ such that,
 for all but finitely many $z \in \ponec$, we have that $\varphi$ maps $\rho^{-1}(z)$ onto $\rho^{-1}(\widetilde{\varphi}(z))$.
The map $\widetilde{\varphi} : \ponec \to \ponec$ obtained from $\varphi$ is called the {\sf reduction of $\varphi$} (e.g.~see~\cite[Section~2.3]{SilvermanDynamicsGTM}).
We will recover dynamical systems acting on $\ponec$ from rational maps acting on $\ponel$ via (variations of) this reduction procedure.

The map   $\widetilde{\varphi}$ may be easily computed after writing $\varphi$ as a quotient of polynomials $P$ and $Q$ in $\mathfrak{O}_\L[z]$, so that at least one of the coefficients of these polynomials is a unit in $\mathfrak{O}_\L$ (i.e. has absolute value $1$). Passing to the quotient $\C[z] \equiv (\mathfrak{O}_\L / \mathfrak{M}_\L) [z]$ we obtain polynomials $\widetilde{P}$ and $\widetilde{Q}$ which  may have a non-constant greatest common divisor $H$. Denote  the roots of $H$ by $\cH(\varphi)$ (compare with ``holes'' in~\cite{DeMarcoIteration}). 
It follows that
$\widetilde{\varphi} = R/S$ where  $R, S \in \C[z]$ are such that $\widetilde{P}= R \cdot H$
and $\widetilde{Q} = S \cdot H$, under the agreement that if $S=0$, then $\widetilde{\varphi}\equiv \infty$  (e.g.~see~\cite[Section~2.3]{SilvermanDynamicsGTM}). Moreover, provided that $\widetilde{\varphi}$ 
is not constant, $\varphi(\rho^{-1}(z)) = \rho^{-1}(
\widetilde{\varphi}(z))$ if $z \notin \cH(\varphi)$ and  $\varphi(\rho^{-1}(z)) = \ponel$ otherwise (see~\cite[Proposition~2.4]{riveratesis03}).


Reduction is, roughly speaking, the algebraic counterpart
of taking the limit as $t \to 0$ of a holomorphic family 
$\mish{f}$:

\begin{lemma}
  \label{lem:12}
  Consider a  degree $d \ge 1$ holomorphic family $\mish{f}$ with associated rational map $\mishl{f} : \ponel \to \ponel$.
Then, as $t \to 0$, 
$$f_t \to \widetilde{\mishl{f}}$$
uniformly on compact subsets of $\ponec \setminus \cH(\mishl{f})$.
\end{lemma}

Although elementary, we include a proof for this observation which is important to relate the Archimedean and non-Archimedean worlds.

\begin{proof}
  We may write  $f_t (z)= P_t(z)/Q_t(z)$ where  $\{P_t\},  \{Q_t\}$ are holomorphic families of complex polynomials
so that the associated maps $\mishl{P}, \mishl{Q}$ can be regarded as elements of $\mathfrak{O}_\L[z]$ with at least one of the coefficients involved being a unit. 
In particular
$P_t$ and $Q_t$ converge as $t\to 0$ (as complex maps of $\ponec$) to $P_0$ and $Q_0$, respectively.
Thus, the quotient $f_t (z) = P_t(z)/Q_t(z)$ converges uniformly to $P_0(z)/Q_0(z)$ for all $z$ outside a neighborhood of the
common roots of $P_0$ and $Q_0$. The lemma follows since away from these roots  $P_0(z)/Q_0(z) = \widetilde{\mishl{P}}(z)/\widetilde{\mishl{Q}}(z)=\widetilde{\mishl{f}}(z)$.
\end{proof}

\subsection{Berkovich projective line over $\L$}
Our aim now is to summarize some properties of the Berkovich projective line $\pberl$ over $\L$ with emphasis on those which  are relevant for this paper.
We will not reproduce here the several equivalent definitions 
of $\pberl$ but rather refer the reader to the
excellent literature already available. For a detailed exposition we refer the reader to~\cite{BakerRumelyBook}, for those with preference for a condensed discussion we suggest~\cite{JonssonDynBer}, while the more akin to the foundational sources will certainly enjoy
 Berkovich's monograph~\cite{MonographBerkovich}.

The Berkovich projective line $\pberl$ is, as a topological space,  a compact, Hausdorff, arcwise connected space which contains
$\ponel$ as a dense subset. It has a non-metric tree structure (e.g. see~\cite[Section~2]{JonssonDynBer}). In particular, there is a unique arc $[x,y]$ between any pair of distinct points $x, y$.

A {\sf direction $w$} at a point $x \in \pberl$ is an equivalence class of the relation in $\pberl \setminus \{ x  \}$ which identifies $y$ and $y'$ if
$]x,y] \cap ]x,y'] \neq \emptyset$. That is, the segments joining $x$ to $y$ and $x$ to $y'$ share an initial portion.
The set formed by all the directions $w$ at $x$ is called the {\sf tangent space at $x$} which will be denoted by $\tang{x}$.
The set of points which represents $w \in \tang{x}$ is denoted by $U_{w}$ (e.g. see~\cite[Appendix B.6.]{BakerRumelyBook} or~\cite[Section~2.1.1]{JonssonDynBer}).
If $x$ is a branched point of $\pberl$, then we say that  $U_{w}$ is a {\sf strict open Berkovich disc}. 
The finite intersection of strict open Berkovich discs is called a {\sf strict basic open set}.

The topology of $\pberl$ coincides with the smallest topology such that all the strict open Berkovich discs are open (e.g. see~\cite[Section~3.6]{JonssonDynBer}).
It is not difficult to check that the sets of the form $U_w$ are (arcwise) connected
and that $\partial U_w = \{ x\}$ for all $w \in \tang{x}$ and all $x \in  \pberl$ (e.g. see~\cite[Section~2.6]{BakerRumelyBook}). 

The points of $\ponel \subset \pberl$ are endpoints in the tree structure of $\pberl$. We refer to these point as {\sf rigid points}, {\sf type I points}, or
{\sf classical points}. 

Points with at least three directions are branched points in the tree structure of  $\pberl$.
Branched points of $\pberl$ are called {\sf type II} points
and will play a central role in our discussions. 
It turns out that at any type II point $x$ there are uncountably many directions. In fact, $T_x \pberl$ is naturally isomorphic to the complex projective
line, that is to $\ponec$ (e.g. see~\cite[Section~3.8.7]{JonssonDynBer}). 

After identification of $\ponel$ with $\L \cup \{\infty\}$ via $[z:1] \mapsto z$, type II points are in one--to--one correspondence with the collection of
closed balls $B \subset \L$ with radius in $| \L^\times | = \exp(\Q)$. 
The point corresponding to the ring of integers $\mathfrak{O}_\L$ is called the {\sf Gauss point} and will be denoted $x_g$.  Directions at the Gauss point
are closely related to the reduction map  $\rho: \ponel \to \ponec$.
In fact, $U_w$ is a direction at $x_g$ if and only if
$U_w$ is the convex hull of $\rho^{-1}(z)$ for some $z \in \ponec$.
Moreover, the induced map $\tang{x_g} \to \ponec$ is an isomorphism of projective lines over $\C$. That is, we may identify $\tang{x_g}$ with $\ponec$ via reduction (e.g. see~\cite[Section~3.8.7]{JonssonDynBer}).

The rest of Berkovich space consists of points of types III and IV that we will not discuss here.

\medskip
The action $\varphi: \ponel \to \ponel$  of any rational map $\varphi \in \L(z)$  extends continuously to $\pberl$.
To ease notation we will also denote the extended map by $\varphi: \pberl \to \pberl$. This continuous extension respects
 compositions and thus iterations of rational maps.
The action $\vphi: \pberl \to \pberl$ is open, surjective, preserves the type of the points (I--IV) and each point has  at most $\deg \varphi$ preimages (see~\cite[Section~4.5]{JonssonDynBer} and/or~\cite[Proposition 2.15, Corollaries 9.9, 9.10]{BakerRumelyBook}).
There is a notion of local degree $\deg_x \varphi$ at any $x \in \pberl$ that coincides with the usual notion at type I points. If we count multiplicities
accordingly, then every point in $\pberl$ has exactly 
$\deg \varphi$ preimages (see~\cite[Section~4.6]{JonssonDynBer} and/or~\cite[Corollary 9.17]{BakerRumelyBook}) ).

A rational map  $\varphi$ induces a {\sf tangent map} $T_x \varphi: \tang{x} \to \tang{\varphi(x)}$ at every point $x \in \pberl$. 
Namely, there exists a connected neighborhood $V$ of $x$ such that, for all directions 
$w \in  \tang{x}$, we have that $\varphi(U_w \cap V) \subset U_{w'}$ for some $w' \in  \tang{\varphi(x)}$  
The direction $w'$ only depends on $w$ and not on the choice of $V$ (see~\cite[Theorem 9.26]{BakerRumelyBook} and/or~\cite[Corollary 2.13]{JonssonDynBer}).
The  {\sf tangent map} $T_x \varphi: \tang{x} \to \tang{\varphi(x)}$ is defined by  $T_x \varphi (w) = w'$.
With this notation, $\varphi (U_w) = U_{w'}$ or  $\varphi (U_w) = \pberl$. In the former case we say that $w$ is a {\sf good direction at $x$} and 
in the latter we say that $w$ is a {\sf bad direction at $x$}.

According to Rivera-Letelier~\cite[Section 2]{riveratesis03} (for a proof in the language of Berkovich spaces closer to our exposition see~\cite[Sections 2.3,~9.1]{BakerRumelyBook}) we have the following:

\begin{proposition}
  \label{pro:3}
Let $\varphi: \pberl \to \pberl$ be a rational map of degree at least $1$. Then $\varphi$ fixes the Gauss point $x_g$ if and only if $\deg \tilde{\varphi} \ge 1$.

Assume that  $\varphi (x_g) = x_g$.  After identifying $\tang{x_g}$ to $\ponec$ via reduction,  the following holds:

\begin{enumerate}
\item $\deg_{x_g} \varphi = \deg \widetilde{\varphi}$. 

\item $T_{x_g} \varphi (w) = \widetilde{\varphi}(w)$ for all $w \in \ponec$.

\item For all $w \in \cH(\varphi)$, we have that  $\varphi (U_w) = \pberl$.
Let $w'= \widetilde{\varphi}(w)$. Then there exists $\delta \ge 1$ such that
every point in $U_{w'}$ has $\delta + \deg_w \widetilde{\varphi}$ preimages in $U_w$ and every point not in 
 $U_{w'}$ has $\delta$ preimages in $U_w$, counting multiplicities.

\item For all $w \in \ponec \setminus \cH(\varphi)$, we have that $\varphi (U_w)= U_{w'}$  where $w'=\widetilde{\varphi}$.
The degree of $\varphi: U_w \to U_{w'}$ is well defined and coincides with
$\deg_w \widetilde{\varphi}$.

\end{enumerate}
\end{proposition}

The group of affine linear transformations 
acts transitively on type II points of $\pberl$. In fact, consider closed balls  $B \subset \L$ and  $B' \subset \L$ with radii in the value group.
Then any affine linear transformation mapping the ball $B \subset \ponel$ onto the ball $B' \subset \ponel$,   maps the type II point corresponding to $B$ onto that corresponding to  $B'$.
Thus the group $\operatorname{PGL} (2, \L)$ of M\"obius transformations also acts transitively on type II points (c.f.~\cite[Corollary 2.13]{BakerRumelyBook}) and, in view of Proposition~\ref{pro:3},
the stabilizer of the Gauss point $x_g$ is $\operatorname{PGL} (2, \mathfrak{O}_\L)$.
Moreover, for all $\gamma \in \operatorname{PGL} (2, \mathfrak{O}_\L)$, we have  that
$\widetilde{\gamma} \in \operatorname{PGL} (2, \C)$ and $\rho \circ \gamma = \widetilde{\gamma} \circ \rho$ where $\rho: \ponel \to \ponec$ denotes the reduction map. 
In particular, $\gamma$ maps the direction containing $\rho^{-1}(z)$ onto the direction containing $\rho^{-1} (\widetilde{\gamma}(z))$, for all $z \in \ponec$.
Similarly, if $\gamma: \pberl \to \pberl$ is any linear fractional transformation and $x$ is a type II point, then $\gamma$ induces a bijection between $\tang{x}$ and $\tang{\gamma(x)}$. In fact, for all directions $w \in \tang{x}$ we have that $ \gamma(U_w)=U_{w'} $ for some $w' \in \tang{\gamma(x)}$
and the assignment $w \mapsto w'$ is an isomorphism between complex projective lines.

To understand the action of a  rational map $\varphi: \pberl \to \pberl$ on type II points and on their corresponding tangent spaces, it is convenient to change coordinates 
in order to work at the Gauss point. Namely, consider two M\"obius transformations $\eta, \gamma : \pberl \to \pberl$ and the points
$x = \eta (x_g), y = \gamma (x_g)$. Then, $y = \varphi (x)$ if and only if the rational map $\psi = \gamma^{-1} \circ \varphi \circ \eta$ fixes the Gauss point.
In this case, the action of $T_x \varphi$ is given by $\widetilde{\psi}$, in the corresponding coordinates.








\subsection{Rescalings and Berkovich periodic points}
Rescalings  of a holomorphic family $\mish{f}$ and type II periodic points of the associated map $\mishl{f}: \pberl \to \pberl$ are closely related.

\begin{proposition}
  \label{pro:2}
  Let $\mish{f}$ be a degree $d \ge 2$ holomorphic family and $\mish{M}$ a moving frame.
Consider the rational map $\mishl{f}: \pberl \to \pberl$ associated to $\mish{f}$ and the automorphism
$\mishl{M}:\pberl \to \pberl$ associated to $\mish{M}$.

Then, for all $\ell \ge 1$, the following assertions are equivalent:
\begin{itemize}
\item[(1)] There exists a rational map $g : \ponec \to \ponec$ 
of degree at least $d'\ge1$ such that, as $t \to 0$, 
$$M^{-1}_t \circ f_t^\ell \circ M_t (z) \to g(z)$$
uniformly on compacts subsets of $ \ponec$ with finitely many points removed.
\item[(2)] $\mishl{f}^\ell (x) = x$ where $x = \mishl{M}(x_g)$ and $\deg_x \mishl{f}^\ell  = d'\ge 1$.
\end{itemize}
In the case that (1) and (2) hold, $T_x \mishl{f}^\ell : \tang{x} \to \tang{x}$ is conjugate via a $\ponec$-isomorphism to $g:\ponec \to \ponec$.
\end{proposition}

\begin{proof}
(1) $\implies$ (2). Note that the rational map associated to $\{M^{-1}_t \circ f_t^\ell \circ M_t\}$ is $\mishl{F}=\mishl{M}^{-1} \circ  \mishl{f}^\ell \circ \mishl{M}$.
From Lemma~\ref{lem:12} the reduction of $\mishl{F}$ has degree at least $d'\ge 1$. Thus, $\mishl{F} (x_g) = x_g$ by Proposition~\ref{pro:3}, and $\deg_{x_g} \mishl{F} =d'$.
It follows that $\mishl{f}^\ell (\mishl{M}(x_g)) = \mishl{M}(x_g)$. Moreover, $\deg_x \mishl{f} = \deg_{x_g} \mishl{F}$ for $x = \mishl{M}(x_g)$ since 
the local degree remains unchanged under pre and post-composition by automorphisms.

(2) $\implies$ (1). From Lemma~\ref{lem:12}, as $t \to 0$, outside a finite set,  $M^{-1}_t \circ f_t^\ell \circ M_t (z) \to \widetilde{\mishl{F}}$ where $\mishl{F}=\mishl{M}^{-1} \circ  \mishl{f}^\ell \circ \mishl{M}$. By Proposition~\ref{pro:3},  $\widetilde{\mishl{F}}$ has degree $\deg_x \mishl{f}^\ell \ge 1$, since $\mishl{f}^\ell (x) = x = \mishl{M}(x_g)$.
\end{proof}

\begin{corollary}
  \label{cor:3}
  Given  a degree $d \ge 2$ holomorphic family $\mish{f}$ and  a moving frame $\mish{M}$, consider the 
subset $S$ of $\N$ formed by all integers $\ell$ such that (1) of Proposition~\ref{pro:2} holds for some $g$ of with $\deg g \ge 2$.
Then $S$ is empty or $S = q \cdot \N$ for some $q \ge 1$.
\end{corollary}

\begin{proof}
Let $x = \mishl{M}(x_g)$. From the previous proposition, 
  if $\ell \in S$, then 
 $\mishl{f}^\ell (x) = x$ and $\deg_x \mishl{f}^\ell \ge 2$.
Thus $x$ is periodic under $\mishl{f}$, say of period $q$, and $S = q \cdot \N$.
\end{proof}

\subsection{Equivalent moving frames and dynamically dependent rescalings}
\label{sec:equiv-moving-fram}
Equivalent moving frames from the viewpoint of Berkovich projective line are characterized as follows.

\begin{lemma}
  \label{lem:3}
  Let $\mish{M}$ and $\mish{L}$ be moving frames. 
Denote by  $\mishl{M}$ and  $\mishl{L}$ (resp.) the associated M\"oebius 
transformations, acting on $\pberl$.
The moving frames $\mish{M}$ and $\mish{L}$ are equivalent if and only if 
$\mishl{M}(x_g) = \mishl{L}(x_g)$. 
\end{lemma}

\begin{proof}
  By Definition~\ref{def:4}, $\mish{M}$ and $\mish{L}$ are equivalent if and only if 
$M_t^{-1} \circ L_t \to M$ for some M\"obius transformation $M \in \rat^1_\C$. 
By Lemma~\ref{lem:12}, this occurs if and only if the reduction of  $\mishl{M}^{-1} \circ \mishl{L}$
has degree $1$ which, in view of Proposition~\ref{pro:3}, occurs if and only if $\mishl{M}^{-1} \circ \mishl{L} (x_g) = x_g$
and the lemma follows.
\end{proof}

We conclude that the  equivalence classes $\cB$ of moving frames are naturally embedded into the set of type II points of  $\pberl$. That is, we have a well defined and injective map:
$$\begin{array}{cccc}
  \iota:& \cB & \to & \pberl \\
  &[\{M_t\}] & \mapsto & \mishl{M}(x_g).
\end{array}
$$
Moreover, our next result shows that $\iota (\cB)$ is invariant under the action associated to holomorphic families.

\begin{lemma}
  \label{lem:5}
  Let $\mish{f}$ be a holomorphic family of degree $d \ge 1$.
  If $\mish{M}$ is a moving frame, then there exists a moving frame $\mish{L}$
such that the following equivalent conditions hold:
\begin{enumerate}
  \item $\mishl{f} \circ \mishl{M} (x_g) = \mishl{L}(x_g)$.
  \item $L^{-1}_t \circ f_t \circ M_t (z) \to g (z),$ 
for some non-constant rational map $g : \ponec \to \ponec$, with uniform convergence on compact subsets outside some finite set.
  \end{enumerate}
In this case, the reduction of $\mishl{L}^{-1} \circ \mishl{f} \circ \mishl{M}$ coincides with $g$.
\end{lemma}

\begin{proof}
  The fact that conditions (1) and (2) are equivalent follows at once from Lemma~\ref{lem:12} and Proposition~\ref{pro:3}.

 To prove the existence of the moving frame $\mish{L}$ consider the holomorphic family  $\{ g_t = f_t \circ M_t \}$. 
It is sufficient to show that there exists a moving frame $\mish{N}$ such that
$\mishl{N} \circ \mishl{g} (x_g) = x_g$.
Write $g_t(z)$ as  $P_t (z) / Q_t (z)$ where 
\begin{eqnarray*}
  P_t(z) &=& a_0 (t) z^d + \cdots  + a_d (t),\\
  Q_t(z) &=& b_0 (t) z^d + \cdots  + b_d (t),
\end{eqnarray*}
for some holomorphic functions $a_j(t), b_j(t)$ defined on a neighborhood of $t=0$.
Taking $A_t (z) = t^m z$ or $z/t^m $ for some $m \ge 0$, and replacing $g_t$ by $A_t \circ g_t$, we may assume
that $a_j (0) \neq 0$ for some $j$ and $b_\ell (0) \neq 0$ for some $\ell$.
It follows that,
$\widetilde{\mishl{P}} \neq 0$ and $\widetilde{\mishl{Q}} \neq 0 $. 
If $\widetilde{\mishl{P}} \neq c \cdot \widetilde{\mishl{Q}}$ for all $c \in \C$, then 
$\widetilde{\mishl{g}}$ has degree at least $1$ and, from Proposition~\ref{pro:3}, we have that ${\mishl{g}} (x_g) = x_g$ as required.
Otherwise, $\widetilde{\mishl{P}} = c \cdot \widetilde{\mishl{Q}}$ for some $c \in \C$. Hence, all the coefficients
of $P_t (z) - c Q_t (z)$ vanish at $t=0$. Take $c_0 \in \C$ such that $\max \{ | \mishl{a_j} - c_0 \mishl{b_j}| \}$ attains its minimum.
Equivalently, given  $c \in \C$ let $\ell_c \ge 0$ be the largest integer such that  $a_j (t) - c b_j(t)$ vanishes at $t=0$ with order at least $\ell_c$, for all $j$.
Such $\ell_c$ exists since $g_t$ has degree $d$ for all $t \neq 0$. 
It follows that we may choose $c_0$ such that the corresponding $\ell_{c_0}$ is maximal. Now let $N_t (z) = (z- c_0)/t^{\ell_0}$ and observe that
 $\mishl{N} \circ \mishl{g} = (\mishl{P} - c_0 \mishl{Q})/(t^{\ell_0} \mishl{Q})$ has reduction of degree at least $1$.
Thus, from Proposition~\ref{pro:3}, we have that $\mishl{N} \circ {\mishl{g}} (x_g) = x_g$ as required. 
\end{proof}

By Lemma~\ref{lem:3}, the moving frame $\mish{L}$ of the previous lemma is unique modulo equivalence.
Moreover, $[\mish{L}]$ is independent of the representative of
 $[\mish{M}]$.  We write $$\mish{f} ([\mish{M}]) = [\mish{L}].$$ 
Note that $$\mish{f} ( \mish{g} ([\mish{M}])) = \{f_t \circ g_t \} ([\mish{M}]),$$
for all holomorphic families $\mish{g}$.
Also,
$$\mishl{f} \circ \iota ([\mish{M}]) = \iota \circ \mish{f} ([\mish{M}])$$ 
for all $[\mish{M}] \in \cB$.

Thus, as anticipated in the introduction, we may talk about orbits of equivalence classes of moving frames under $\mish{f}$.
Rescalings within the same periodic orbit will be regarded as the same rescaling for $\mish{f}$.

\begin{definition}
  \label{def:8}
  Consider a holomorphic family $\mish{f}$. Assume that $\mish{M}$ and $\mish{L}$ are rescalings for $\mish{f}$.
We say that $\mish{M}$ and $\mish{L}$ are {\sf dynamically dependent rescalings} if $\mish{f^\ell} ([\mish{M}]) = [\mish{L}]$, 
for some $\ell \ge 0$.
\end{definition}

Let us take a closer look to  rescalings associated to  dynamically dependent rescalings.
That is, assume that the moving frame $\mish{M}$ is a rescaling of period $q$ for  $\mish{f}$.
For $\ell=0, \dots, q-1$ let $\mish{M^{(\ell)}}$ be a representative of $\mish{f^\ell} ([\mish{M}])$, subscripts modulo $q$.
Denote by
 $g_0, \cdots, g_{q-1}$ (subscripts modulo $q$) the non-constant complex rational maps such that
$$\left(M^{(\ell)}_t \right)^{-1} \circ f_t \circ M^{(\ell-1)}_t \to g_{\ell-1}$$
as $t \to 0$, outside a finite set.
It follows that for all $\ell$, the moving frame $\mish{M^{(\ell)}}$ is a rescaling of period $q$ for $\mish{f}$ with rescaling limit
$$G_\ell =  g_{\ell-1} \circ \dots \circ g_{\ell+1} \circ g_{\ell}.$$
Note that the degree of $G_\ell$ is independent of $\ell$. Also, either $G_\ell$ is  postcritically finite for all $\ell$, or none of these maps
$G_\ell$ is postcritically finite.

\medskip
We will deduce Theorem~\ref{ithr:3}
from a result  regarding Berkovich space dynamics.
In order to make this deduction it is convenient to record 
a direct consequence  of the above discussion.

\begin{corollary}
  \label{cor:1}
Consider a holomorphic family $\mish{f}$ of degree $d \ge 2$.
If $\mish{M'}$ and $\mish{M}$ are dynamically independent rescalings,
then  the periodic orbits of $x'=\iota([\mish{M'}])$ and $x = \iota([\mish{M}])$, under iterations
of  $\mishl{f}$, are distinct.
\end{corollary}

\section{Dynamics in the Berkovich projective line and rescalings}
\label{sec:ber2res}

The previous section stresses the importance of type II periodic orbits in the study rescaling limits.
Here we will be concerned with the critical points in the ``basin'' of these periodic orbits.
More precisely, given a rational map $\varphi : \pberl \to \pberl$ of degree $d \ge 2$ with a periodic orbit $\cO$, we say that
 the {\sf basin of $\cO$} is the interior of the set of
points $x \in \pberl$ such that, for all neighborhoods $U$ of $\cO$,
the orbit of $x$ is eventually contained in $U$ (i.e. there exists
$n_U$ such that $\varphi^n(x) \in U$ for all $n \ge n_U$).

Periodic orbits of rational maps $\varphi : \pberl \to \pberl$ may be 
classified according to whether they belong to the ``Julia set'' $J(\vphi)$ or to the ``Fatou set'' $F(\vphi)$.
A point $x \in \pberl$ lies in the {\sf Julia set} if
 for all neighborhoods $U$ of $x$ we have that $\cup \varphi^n(U)$ omits at most two points of $\pberl$.
As usual, the complement of the Julia set, is the Fatou set. (See~\cite[Section~10.5]{BakerRumelyBook}.)
A periodic point $x \in \pberl \setminus \ponel$ of period $q$ lies in the Julia set $J(\varphi)$ if either 
$\deg_x \vphi^q \ge 2$ or there exist a bad direction of $\vphi^q$ at $x$ with infinite forward orbit under $T_x \vphi^q$ (see~\cite[Theorem~2.1]{KiwiPuiseuxQuadratic}). 
In the first case, $x$ is called a {\sf repelling periodic point}. In the latter case, when $\deg T_x \varphi =1$, we say that 
$x$ is an {\sf indifferent periodic point}.
Every non-rigid periodic point that belongs to the Julia set is of type II (see ~\cite[Section~10.7]{BakerRumelyBook}).
It is worth to emphasize, that despite its name, every repelling periodic point $x$ of type II, has non-empty basin.
In fact, without loss of generality assume that $x$ is fixed under $\varphi$.
Let $v \in T_x \pberl$ be a good direction such that its iterates under $T_x \varphi$ are pairwise distinct good directions.
Such $v$ always exists since $\deg T_x \varphi \ge 2$ and the residue field is uncountable. It follows $\{ \varphi^n (U_v) \}$ are pairwise disjoint Berkovich open disks 
with boundary point $x$. Hence, every point in $U_v$ belongs to the basin of $x$. 

According to the previous section, rescalings correspond to repelling periodic orbits. Our next result
explores under which conditions the basin of a periodic orbit does not contain rigid critical points.

\begin{theorem}
\label{FinitelyManyNonPCFTheorem}
Consider a  rational map $\vphi : \pberl \to \pberl$ with $\deg
\vphi \geq 2$.  
Let $\cO$ be a type II periodic orbit of period $q \ge 1$ of $\vphi$.
Assume that the basin of $\cO$ is free of type I critical points.
Then, for all $x \in \cO$, every bad direction of $\vphi^q$ at $x$ has finite forward orbit under $T_x \vphi^q$.
Moreover, if $\deg_x \vphi^q \ge 2$, then
 $T_x \vphi^q$ is postcritically finite.  
\end{theorem}

The proof relies on the fact that  any bad direction contains a rigid critical point and its ``image'' contains a critical value.
This is a particular and simple result about the ``ramification locus'' of a rational map which we state and prove in the lemma below.
See the more general work by Faber for a detailed study of the ramification locus~\cite{FaberRamificationI,FaberRamificationII}.

\begin{lemma}
\label{CriticalValueLemma}
  Let $\vphi: \pberl \to \pberl$ be a rational map of degree at least $2$.
Consider a type II point $x \in \pberl$. 
Let  $v \in \tang{x}$ and consider  $w = T_x\varphi (v) \in \tang{\varphi(x)}$. Denote by $U_v$ and $U_w$ the corresponding strict open Berkovich discs.

If $\vphi: U_v \to \vphi(U_v)$ is not injective, then there exists a rigid critical point $c \in U_v$ such that the corresponding critical value $\vphi(c) \in U_w$. 
\end{lemma}

Note that according to Proposition~\ref{pro:3},  we have that $\vphi: U_v \to \vphi(U_v)$ is not injective if and only if 
$v$ is a bad direction or $\deg_v T_x \varphi \ge 2$.

\begin{proof}
If $v$ is a good direction at $x$, then the result follows directly from the Non-Archimedean Rolle's Theorem (see~\cite[Lemma 10.43]{BakerRumelyBook}).

Assume that $v$ is a bad direction. According to Proposition~\ref{pro:3} (3),  there exists at least one preimage of $\vphi(x)$ in $U_v$.
Denote by $V$ the connected component  $U_v \setminus \vphi^{-1}(\vphi(x))$ such that $x \in \partial V$
and observe that $V$ is strictly contained in $U_v$.
It follows that $V$ is a strict basic open set  and that $\vphi (V) = U_w$, since $\vphi (V)$ is open and closed in $U_w$.
Let $\cA_V$ be the skeleton of $V$ (i.e. the convex hull of $\partial V$). 
Note that $\vphi(\cA_V)$ is also a (finite) tree~\cite[Corollary 9.36]{BakerRumelyBook}. 
Since all the endpoints of $\cA_V$ are mapped onto $\vphi(x)$, there must
exist a point $y \in \cA_V$ with valence $k \geq 2$ which is mapped to an endpoint of $\vphi(\cA_V)$.
All the directions at $y$ which intersect $\cA_V$ must map (under $T_y \vphi$) to the unique direction at
$\vphi(y)$ which intersects $\vphi(\cA_V)$.
Hence, $y$ is a type II point, since $T\varphi$ is injective at type I, III, IV points~\cite[Corollary 9.20]{BakerRumelyBook}. Also, $\deg T_y \vphi \geq k \geq 2$ and $T_y \vphi : T_y \pberl \to T_{\vphi(y)} \pberl$ must have at least two critical values.
So there exists a critical value $b$ of $ T_y \vphi$ which is a direction at $\vphi(y)$ distinct than the one determined by   $\vphi(\cA_V)$.
The corresponding critical point  $c$ of $ T_y \vphi$ is a direction at $y$ 
distinct from those determined by the elements of $\partial V$. Thus, the corresponding disc $U_c$ is completely contained in $V$. Therefore, 
$\vphi(U_c) = U_b$ and $\vphi:U_c \to U_b$ is of degree at least $2$. By the Non-Archimedean Rolle's Theorem~\cite[Lemma 10.43]{BakerRumelyBook}, the disc $U_c$ contains
a critical point, $U_b$ contains a critical value and the lemma follows.
\end{proof}

\begin{proof}[Proof of Theorem~\ref{FinitelyManyNonPCFTheorem}]
We proceed by contradiction and suppose that
there exists a  direction $w$ at $T_x \pberl$ with infinite forward orbit under $T_x \vphi^q$
that contains a rigid critical point of $\vphi^q$.
It is sufficient to show that there exists a rigid critical point of $\varphi$ in the basin of the periodic orbit $\cO$ of $x$.
Hence, it is also sufficient to show that the forward orbit under $\vphi^q$ of a rigid critical point of $\vphi^q$ converges to $x$.

We claim that the disc in the direction $T_x \vphi^{nq}(w)$ contains a point in the forward
orbit of a rigid critical point of $\vphi^q$ for all $n \geq 0$.  We proceed
by induction.  Suppose that the disc $U_n$ in the direction $T_x \vphi^{nq}(w)$ contains a point in
the forward orbit of a critical point of $\vphi^q$.  If $T_x
\vphi^{nq}(w)$ is a good direction at $x$ for $\vphi^q$, then
$\vphi^q(U_n) = U_{n+1}$ contains a
critical orbit element.  If $T_x \vphi^{nq}(w)$ is a bad direction at
$x$ for $\vphi^q$, then the previous lemma guarantees that the disc $U_{n+1}$ associated to the 
direction $T_x \vphi^{(n+1)q}(w)$ contains a critical value of $\vphi^q$.

Now, for $n$ sufficiently large, say for $n \geq n_0$, the direction
$T_x \vphi^{nq}(w)$ is always a good direction. Moreover, for some rigid critical point
$c$ of $\vphi^q$ and some integer $\ell \ge 1$ we have that 
$\vphi^{\ell q}(c) \in U_{n_0}$. Therefore, $\vphi^{nq + \ell q}(c) \to x$ as $n \to \infty$, since $\vphi^{nq + \ell q}(c) \in U_{n+n_0}$ for all $n \ge 0$
and,  given a neighborhood $V$ of $x$, there are at most finitely many discs $U_{n+n_0}$ not contained in $V$.
\end{proof}

For the sake of simplicity, if $\cO$ is a periodic orbit of period $p$, then we let $T_\cO \varphi$ denote $T_\z \varphi^p$ for some $\z \in \cO$.
Note that $T_\cO \varphi$ depends on $\z \in \cO$, and not necessarily these maps are conjugate for different choices of $\z$. However, there are some
properties that are independent of the choice, such as being postcritically finite or conjugate to a monomial.

\begin{corollary}
\label{cor:2}
  Consider a rational map $\varphi : \pberl \to \pberl$ of degree $d \geq 2$. 
  Then $\varphi$ has at most $2d-2$ repelling type II periodic orbits $\cO$ such that  $T_x \vphi^q$ is not postcritically finite for all $x \in \cO$ where $q$ is the period of $\cO$. 
\end{corollary}

The examples, known to the author, where $\vphi$ has infinitely many non-rigid repelling periodic points have the property that for all 
these periodic points $x$ of $\vphi$, with at most finitely many exceptions, the action on the tangent space $T_x \varphi^q$ is conjugate to $z^{\pm k}$ for some $k \geq 2$, where $q$ is the period of $x$. It is natural to ask whether 
at most finitely many repelling periodic orbits are such that  $T_x \vphi^q$ is not ``monomial'', for all $x \in \cO$ where $q$ is the period of $\cO$.

The dynamics of degree $2$ rational maps $\varphi: \pberl \to \pberl$ was studied in detail  in~\cite{KiwiPuiseuxQuadratic}.
We reproduce below  the statement of Theorem~1~in~\cite{KiwiPuiseuxQuadratic}  which describes the portion of the Julia set outside $\ponel$. 
Of particular interest for us will be the description of the possible configurations for type II repelling periodic orbits (parts (3) and (4)).

\begin{theorem}
  \label{thr:periodic}
  Let $\vphi: \pberl \to \pberl$ be a quadratic rational map without a non-rigid repelling fixed point. Then one and exactly one of the following holds:
  \begin{enumerate}
  \item $\juliavphi \setminus \ponel = \emptyset$.
  \item $\juliavphi \setminus \ponel$ is the grand orbit of an indifferent periodic orbit. 
  \item $\juliavphi \setminus \ponel$ is the grand orbit of one repelling periodic orbit $\cO$ of period $q \ge 2$. 
For all $x \in \cO$, the  map $T_x \vphi^q : \ponec \to \ponec$ is a quadratic rational map with a multiple fixed point. 
  \item $\juliavphi \setminus \ponel$ is the grand orbit of two distinct periodic orbits $\cO, \cO'$ of periods $q,  q' \ge 2$, respectively, where $q' > q$. 
The  map $T_x \vphi^q : \ponec \to \ponec$ is a  quadratic rational map with a multiple fixed point, for all $x \in \cO$.
The map  $T_x \vphi^{q'} : \ponec \to \ponec$ is (modulo choice of coordinates) a quadratic polynomial, for all $x \in \cO'$.
  \end{enumerate}
\end{theorem}

\begin{proof}[Proof of Theorem~\ref{ithr:3}]
Suppose that $\mish{M^{(1)}}, \dots, \mish{M^{(N)}}$ are pairwise dynamically independent rescalings for $\mish{f}$ of periods $q_1, \dots, q_N$, with 
 postcritically infinite rescaling limits. Let $x_j = \mishl{M}^{(j)} (x_g)\in \pberl$ for $j=1, \dots, N$. 
From Proposition~\ref{pro:2}, we have that $T_{x_j} \mishl{f}^{q_j}: \ponec \to \ponec$ is not postcritically finite, for all $j=1, \dots, N$. In view of Corollary~\ref{cor:1},
the points $x_1, \dots, x_N$ lie in pairwise distinct periodic orbits of $\mishl{f}$. It follows that $N \leq 2d -2$, by Corollary~\ref{cor:2}.

For quadratic rational maps (i.e., when $d=2$) we have that $N \leq 2$. If $N=2$, then (4) of Theorem~\ref{thr:periodic} holds for $\mishl{f}$, and from Proposition~\ref{pro:2} we conclude that (1) of  Theorem~\ref{ithr:3} holds. If $N =1$ and $q_1 > 1$, then (2) of Theorem~\ref{thr:periodic} holds for $\mishl{f}$ and, similarly, we obtain that (2) of  Theorem~\ref{ithr:3} holds.
\end{proof}

\section{Sequential rescalings}
\label{sec:sequential}
In this section we discuss the basic properties of sequential rescalings needed in order to prove Theorem~\ref{ithr:2}.
We start with the discussion of sequential limits of rational maps which is closely related to De Marco's article~\cite{DeMarcoIteration} but we use a slightly different 
terminology. Then we continue to study the definitions of equivalence and dynamical independence of sequences of frames to finish the section
by introducing a convenient choice
of a rescaling sequence given a rescaling limit.

\subsection{Sequences and reduction}
Recall that we identify $\ratdc$ with a subset of $\P^{2d+1}_\C$.
For convenience we use homogeneous coordinates $[x:y]$ in $\ponec$.
Given any $f=[a_0: \cdots:a_d:b_0:\cdots:b_d] \in \P^{2d+1}_\C$
let 
$$P(x,y)= a_0 y^d + a_1 x y^{d-1} + \cdots + a_d x^ d,$$
$$Q(x,y)=b_0 y^d + b_1 x y^{d-1} + \cdots + b_d x^ d.$$
Consider a homogeneous polynomial $H = \gcd (P,Q)$ of degree at most $d$.
Write
$$P= H \cdot \tilde{P}, \, Q = H \cdot \tilde{Q}.$$
Let $$\cH (f) = \{ [x:y] \in \ponec \mid H(x,y)=0 \}$$
and call $\cH(f)$ the {\sf holes} of $f \in \P^{2d+1}_\C$.
Also, consider $\tilde{f}: \ponec \to \ponec$ the well defined rational map
of degree $d - \deg H$ given by $$\tilde{f} ([x:y]) = [\tilde{P}(x,y): \tilde{Q}(x,y)].$$
We call $\tilde{f}$ the {\sf reduction of $f$} (compare with~\cite{DeMarcoIteration}). 

Note that for a holomorphic family $\mish{f}$ we defined the reduction $\tilde{\mishl{f}}$ of the associated rational map $\mishl{f}: \ponel \to \ponel$. 
It follows that $\tilde{\mishl{f}}$ agrees with the reduction  $\tilde{f_0}$ of the point $f_0 \in \P^{2d+1}_\C$.

\begin{lemma}
  \label{lem:9}
  Consider a sequence $\seqfn \subset \P^{2d+1}_\C$ converging to $f \in \P^{2d+1}_\C$. Then $f_n \to \tilde{f}$ uniformly on compact subsets of $\ponec \setminus \cH(f)$, where $\tilde{f}$ denotes the reduction of $f$ and $\cH(f)$ the  holes of $f$.
\end{lemma}

We omit the straightforward proof of the previous lemma (e.g., see the proof of Lemma~\ref{lem:12}).

\begin{lemma}
  \label{lem:11}
  Consider sequence $\seqfn \subset \ratd \subset \P^{2d+1}_\C$ 
and $\{ g_n \} \subset \rat^{d'}_\C \subset \P^{2d'+1}_\C$ converging to $f \in \P^{2d+1}_\C$ and $g \in \P^{2d'+1}_\C$, respectively, 
If $\deg \tilde{f} \ge 0$ and $\deg \tilde{g} \ge 1$, then 
$$ f_n \circ g_n \to \tilde{f} \circ \tilde{g}$$ 
uniformly on compact subsets of
$\ponec \setminus S$ where $S$ is the finite set given by  $$S = \tilde{g}^{-1} (\cH(f)  ) \cup \cH (g).$$
\end{lemma}

Again the proof is straightforward after observing that $\tilde{g}$ is finite--to--one and onto.

\subsection{Basic properties of rescalings}
Rescaling limits and dynamical dependence behave nicely under equivalence of sequences of frames.

\begin{lemma}
  \label{lem:10}
  Suppose that $\seqMn$ and $\seqLn$ are equivalent rescalings for $\seqfn$. Then their rescaling period coincides and the rescaling limits are M\"obius conjugate.
\end{lemma}

\begin{proof}
  If $M_n^{-1} \circ L_n \to M \in \rat^1_\C$, then 
$$M_n^{-1} \circ f_n^q \circ M_n =   M_n^{-1} \circ L_n \circ L_n^{-1} \circ f_n^q \circ L_n \circ L_n^{-1} \circ M_n$$
which in view of Lemma~\ref{lem:11} converges to $M^{-1} \circ g \circ M$ outside a finite set, provided that 
$\seqLn$ is a rescaling of period $q$ with limit $g$ for $\seqfn$. 
\end{proof}

Along similar lines one may fill the omitted proof of the following result.

\begin{lemma}
  \label{lem:2}
  Suppose that $\seqMn$ and $\seqMnprime$ are dynamically dependent rescalings for $\seqfn$.
Assume that $\seqLn$ and $\seqLnprime$ are equivalent to $\seqMn$ and $\seqMnprime$ respectively. Then $\seqLn$ and $\seqLnprime$ are dynamically dependent rescalings for $\seqfn$.
\end{lemma}

\subsection{The period of a rescaling}
\label{sec:period-rescaling}
Here we show that the ``period'' of a rescaling $\seqMn$ of $\seqfn$ has the desired property of dividing
 all the iterates of $f_n$ which have a limit after changing coordinates according to $M_n$.

\begin{lemma}
  \label{lem:8}
  Consider a sequence  $\seqfn \subset \ratdc$ for some $d \ge 2$. 
  Assume that there exist $q_0 \ge 1$, $q_1 \ge1$ and rational maps $g_0, g_1$ of degree $\ge1$, such that 
    $$M_n^{-1} \circ f_n^{q_j} \circ M_n \to g_j$$
uniformly on compact subsets outside a finite set. If $q = \gcd(q_0,q_1)$, then 
$$M_n^{-1} \circ f_n^{q} \circ M_n \to g$$
uniformly on compact subsets  outside a finite set, for some rational map $g$ of degree $\ge 1$.
  
   Moreover, if $\deg g_0 > 1$ or $\deg g_1 >1$, then $\deg g >1$.
\end{lemma}

\begin{proof}
  To simplify notation, let $F_n = M_n^{-1} \circ f_n^{q} \circ M_n $.
  Write $q_0 = q_1 \cdot a + r$ with $a, r$ non-negative integers such that
$0 \le r < q_1$. Then, 
  $$F_n^{q_0} = F_n^r \circ F_n^{aq_1}$$
Given any convergent subsequence of $\{ F_{n}^r \}$, say converging to $F$,
it follows that $g_0 = \tilde{F} \circ g_1^a$ (see Lemma~\ref{lem:11}), thus $\deg \tilde{F} \ge 1$. Since $g_1^a$ is onto and $F_n^{q_0} \to g_0$ outside a finite set, 
all convergent subsequences of $\{ F_{n}^r \}$, 
converge to the same map $\tilde{F}$ outside the same finite set. 
From the Euclidean Algorithm it follows that $F_n^q$ converges outside a finite set to a rational map of degree at least $1$. 

If $\deg g_0 > 1$ or $\deg g_1 >1$, then $\deg g >1$, since $g_0$ and $g_1$ are iterates of $g$.
\end{proof}

As an immediate consequence of the above lemma we obtain the following corollary.

\begin{corollary}
  \label{cor:5}
  If $\seqMn$ is a rescaling for $\seqfn$, then the set formed by all $\ell \ge 1$ such that $ M_n^{-1} \circ f_n^{\ell} \circ M_n $
converges uniformly outside a finite set to a rational map of degree at least $2$ is of the form $q \cdot \N$ where
$q$ is the period of $\seqMn$.
\end{corollary}


\subsection{Periodic orbits and  rescalings}
Given a rescaling $\seqMn$ for $\seqfn$, using the repelling periodic orbits of the rescaling limit 
we choose equivalent sequence of frames which  will 
be particularly useful.

\begin{definition}
  \label{def:6}
  Let $\O^3 = \ponec \times \ponec \times \ponec $.
  Given  $\vec{z}=(z_1,z_2,z_3) \in \O^3$ with pairwise distinct entries, we will consistently denote by $\gamma_{\vec{z}}$  the unique M\"oebius transformation such that 
$$z_1 =\gamma_{\vec{z}}(0), z_2 =\gamma_{\vec{z}}(1),
z_3 =\gamma_{\vec{z}}(\infty).$$
\end{definition}

\begin{lemma}
  \label{lem:1}
  Suppose that $\seqMn$ is a rescaling of period $q$ for $\seqfn$ with limit $g$ in $\ponec \setminus S$. 
Let $\cO \subset \ponec \setminus S$ be a period $q'\geq 3$ repelling periodic orbit of $g:\ponec \to \ponec$. Choose three distinct points $z_1, z_2, z_3$ in $\cO$.
Then,  there exist sequences $\{ z_{j,n} \}$ for $j=1,2,3$ such that the following holds:
\begin{enumerate}
\item For all $j,n$ the point $z_{j,n}$ has period $qq'$ under iterations of $f_n$.
\item As $n \to \infty$,  $$M^{-1}_n (z_{j,n}) \to z_j.$$
\end{enumerate}
Moreover, given  $\{ z_{j,n} \}$ for $j=1,2,3$ such that (1) and (2) hold:
\begin{itemize}
\item[(a)]
If $\{ z'_{j,n} \}$ for $j=1,2,3$ are such that (1) and (2) hold, then $z_{j,n} =  z'_{j,n}$ for all $n$ sufficiently large.
\item[(b)] If $\vec{z}_n = (z_{1,n},z_{2,n},z_{3,n})$,
then 
$\{\gamma_{\vec{z}_n} \}$ is equivalent to $\seqMn$.
\end{itemize}
\end{lemma}

\begin{proof}
  For $j=1,2,3$, let $V_j$ be a small neighborhood of $z_j$ such that $g_n = M^{-1}_n \circ f_n^q \circ M_n$ has a unique period $q'$ point $w_{j,n}$ in $V_j$ for all $n$ sufficiently large. We may assume that $V_1,  V_2, V_3$ are pairwise disjoint. These neighborhoods $V_j$ and points $w_{j,n}$ exist since $\cO$ is a repelling periodic orbit of $g$ and the convergence of the holomorphic maps $ M^{-1}_n \circ f_n^{qq'} \circ M_n$ to $g^{q'}$ is uniform on compact subsets of $V_j$. 
Now let $z_{j,n} = M_n (w_{j,n})$ and observe that (1) and (2) holds. Uniqueness of $w_{j,n}$ for large $n$, implies that (a) holds. 
The degree $1$ rational maps $\beta_n = M_n^{-1} \circ \gamma_{\vec{z}_n} $ are such that $\beta_n (p_j) = M_n^{-1} (z_{j,n}) \to z_j$ for $p_1 =0, p_2 =1, p_3 = \infty$.  
We claim that the sequence $\{ \beta_n \}$ converges to $\gamma_{\vec{z}}$.
In fact, assume that  $\beta \in \overline{\rat^1_\C} = \P^3_\C$ is an accumulation point of  $\{ \beta_n \}$.
Thus, a subsequence converges uniformly on the complement of (at most) one point to $\tilde{\beta}$. 
It follows that $\tilde{\beta}$ cannot be constant (it has to take at least two of the values $z_1, z_2, z_3$).
Hence, $\tilde{\beta}$ has lies in $\rat^1_\C$ and the subsequence converges uniformly on $\ponec$.
Therefore, $\tilde{\beta}$ is $\gamma_{\vec{z}}$.
\end{proof}

\section{From sequences  to holomorphic families}

\begin{proposition}
  \label{pro:1}
  Consider a sequence of degree $d$ rational maps $\seqfn$.
  Let $N \in \N$ and assume that for all $j=1, \dots N$ the sequence
$\seqMjn$ is a rescaling of period $q_j$ for $\seqfn$ with rescaling limit
$g_j$ in $\ponec \setminus S_j$. Then there exists a degree $d$ holomorphic family $\mish{f}$ and, for each $j=1, \dots N$, a degree $1$ 
moving frame $\{ M_{j,t} \}$ such that 
$\{ M_{j,t} \}$ is a rescaling for $\mish{f}$ of period $q_j$ and limit $g_j$.

If 
$\{ M_{j,t} \}$ and $\{M_{k,t}\}$ are dynamically dependent
for $\mish{f}$, then 
$\seqMjn$ and $\{ M_{k,n} \}$ are dynamically dependent for 
 $\{ f_n \}$.

Moreover, 
if $\seqfn$ converges to $f \in \partial \rat^d$, then $\mish{f}$ can be chosen so that $f_0=f$. 
\end{proposition}

\begin{proof}
  Throughout the proof we drop the subscripts $\C$ and let $\O^3 = \pone \times \pone \times \pone$. 

We start by changing the rescalings $\{M_{j,n}\}$ for equivalent ones.
Let $$\Delta \subset \O^3$$ 
be the set formed
by all triples $(z_1,z_2,z_3)$ such that at least two of the coordinates
agree. 
From Definition~\ref{def:6},  recall  that
given a triple $\vecz = (z_1, z_2, z_3) \in \O^3 \setminus \Delta$,
we consider the M\"obius transformation $\gamma_\vecz$ for which
$$\gamma^{-1}_\vecz(z)=\dfrac{z_{2}-z_{3}}{z_{2}-z_{1}} \cdot \dfrac{z-z_{1}}{z-z_{3}}.$$

For each $j = 1, \dots, N$, let $\cO_j$ be a repelling periodic orbit of $g_j$ contained in $\pone \setminus S_j$
with period $q'_j \ge 3$. 
Pick three different points $z_{j,1,\infty},z_{j,2,\infty},z_{j,3,\infty}$ in $\cO_j$
and let  $z_{j,1,n},z_{j,2,n},z_{j,3,n}$ be periodic points of
 period $q'_j q_j$ under  $f_n$ such that 
$$M^{-1}_{j,n} ( z_{j,i,n} ) \to  z_{j,i,\infty}$$
for $i=1,2,3$, as in Lemma~\ref{lem:1}.

Let $\vecz_{j,n} =  (z_{j,1,n},z_{j,2,n},z_{j,3,n})$. From Lemma~\ref{lem:1}, it follows that $\{ \gamma_{\vecz_{j,n}} \}$ and $\{ M_{j,n} \}$ are equivalent rescalings for $\seqfn$. By Lemma~\ref{lem:10},
$$\{ \gamma_{\vecz_{1,n}} \}, \dots, \{ \gamma_{\vecz_{N,n}} \}$$
are rescalings for $\seqfn$ with rescaling limits, say $h_1, \dots, h_N$, respectively, such that 
$g_j = L_j^{-1} \circ h_j \circ L_j$ for some M\"oebius transformation $L_j$, for all $j$.
Moreover, in view of Lemma~\ref{lem:9} and by compactness of $\P^{2 d^{q_j} +1}$ we may assume that 
$$\gamma_{\vecz_{j,n}}^{-1} \circ f_n^{q_j} \circ \gamma_{\vecz_{j,n}}$$
converges to a point $H_j \in \partial \rat^{d^{q_j}} \subset \P^{2 d^{q_j} +1}$ such that 
$h_j = \widetilde{H}_j$.
Keep in mind that $z=0, 1$ or $\infty$ is not a hole of $H_j$, since
$\cO_j \subset \pone \setminus S_j$. 
 
\bigskip
Let us construct the appropriate space to find the desired holomorphic families.
For each $j \in \N$, consider the rational map 
$$\begin{array}[h]{cccc}
I_j:  \P^{2d+1} \times \O^3 \times \O^3 & \dashrightarrow & \P^{2d^j+1}\\
 (f,\vecz,\vecw) & \mapsto & \gamma_\vecz^{-1} \circ f^j \circ \gamma_\vecw.
\end{array}
$$ 
which is regular in the complement of 
$\partial\ratd \times \Delta \times \Delta.$ 

Given $p \in \N$, let $X_p$ be the Zariski closure in $\P^{2d+1} \times \O^3 \times \O^3$
of the space formed by all triples $(f, (z_1,z_2,z_3), (w_1,w_2,w_3)) \in \ratd \times \O^3 \times \O^3$
such that  $f^p (z_i) = z_i$ and $f^p(w_i) = w_i$ for all $i=1,2,3.$

Let $p_1=q_1 q_1', \dots, p_N =q_N q_N'$, and 
$$X= X_{p_1, \dots, p_N} \subset \P^{2d+1} \times (\O^3 \times \O^3)^N$$ be the Zariski closure of the algebraic set
formed by all $(f, \vec{z}_1,  \vec{w}_1, \dots, \vec{z}_N,  \vec{w}_N)$ such that
$(f, \vec{z}_j,  \vec{w}_j) \in X_{p_j}$ for $j=1, \dots, N$.

Consider the rational map $$I:X \dashrightarrow \Pi_{j=1}^N ( \P^{2d^{q_j}+1} \times \Pi_{k=0}^{q_j-1}  (\P^{2d^{k}+1} \times \P^{2d^{q_j-k}+1}))$$
given by
$$I (f, \vec{z}_1,  \vec{w}_1, \dots, \vec{z}_N,  \vec{w}_N) = \Pi_{j=1}^N \left(I_{q_j}(f, \vec{z}_j,\vec{z}_j) \times \Pi_{k=0}^{q_j-1} \left(I_k (f, \vec{z}_j,\vec{w}_j) \times I_{q_j-k} ( f,\vec{w}_j,\vec{z}_j)\right)\right).$$
Note that $I$ is regular in the complement of 
$$D= \partial \ratd \times \left( \Delta \times \Delta \right)^N.$$

Consider a resolution 
$\pi_X : \hat{X} \to X$ of the map $I$. That is, $\hat{X}$ is an algebraic variety over $\C$, the map $\pi_X$ is an isomorphism in the complement of
$E=\pi_X^{-1}( D)$ and
$$ \hat{I} = I \circ \pi_{X} $$
originally defined in $\hat{X} \setminus E$ extends  to a regular map defined defined in $\hat{X}$.

\medskip
Now that we have constructed the appropriate spaces we proceed to find the holomorphic family $\mish{f}$ and the moving frames $\{ M_{j,t} \}$.
Since  $(f_n,\vecz_{j,n}, \vecz_{j,n}) \in X_{p_j}$ for all $j$ and $n$,  we have that
$$x_n ({k_1,\dots,k_N}) = (f_n, \vecz_{1,n}, f_n^{k_1}(\vecz_{1,n}), \dots, \vecz_{N,n}, f_n^{k_N}(\vecz_{N,n})) \in X \setminus D,$$
for all integers $k_1, \dots, k_N$ such that $0 \le k_j < q_j$.
Passing to an appropriate subsequence, we may assume that  
$\{ x_n ({k_1,\dots,k_N}) \}$ converges in $\hat{X}$, for all choices of $k_1, \dots, k_N.$

Say the limit of $\{ x_n ({0,\dots,0}) \}$ is $x_\infty \in \hat{X} $. 
Since the projection of $I(x_n(0,\dots,0))$ to the  $\P^{2d^{q_j}+1}$ coordinate converges to $H_j$, for all $j$, we have that the corresponding projection of $\hat{I} (x_\infty)$ is  $H_j$.
Now consider a holomorphic curve $\chi_t \subset \hat{X}$ parametrized by $t \in U$, for some neighborhood $U$ of the origin in $\C$, such that
$\chi_0 = x_\infty$ and $\chi_t$ lies outside the exceptional divisor $E$ of $\hat{X}$ for all $t\neq 0$. 
It follows that the first coordinate of $$t \mapsto \pi_X (\chi_t) = (f_t, \vecz_{1,t},  \vecw_{1,t}, \dots, \vecz_{N,t},  \vecw_{N,t})$$
is a degree $d$ holomorphic family. For each $j=1, \dots, N$ consider the  moving frame
 $$M_{j,t} = \gamma_{\vecz_{j,t}}.$$
Since $\chi_t \to x_\infty$, as $t\to 0$, by continuity of $\hat{I}$, 
 we have that, 
$$M^{-1}_{j,t} \circ f_t^{q_j} \circ M_{j,t} \to H_j,$$
as $t \to 0$, for all $j$, where the convergence is as points in $\P^{2d^{q_j}+1}$.
Thus, outside a finite set we have uniform convergence of $M^{-1}_{j,t} \circ f_t^{q_j} \circ M_{j,t}$ to $h_j$. 
Therefore, we have that $\{  M_{j,t} \circ L_j \}$ is a rescaling of period $q_j$ and limit $g_j$ for $\mish{f}$.

\medskip
To finish the proof it is sufficient to show that if  $\{ M_{j,t} \} $ and $\{ M_{k,t}\}$ are dynamically dependent for $\mish{f}$ for some 
$j \neq k$, then $ \{ M_{k,n} \} $  and $\{ M_{j,n} \}$ are dynamically dependent for $\seqfn$.
Thus, suppose that 
$$\{f^\ell_t\} ([\{M_{j,t}\}])= [\{M_{k,t}\}]$$
for some $\ell$ where $1 \le \ell < q_j$.



Let $\vecz_{j,t} = ( z_{j,1,t},z_{j,2,t},z_{j,3,t})$. Observe that  $z_{j,i,t}$ may be regarded as an element of $\ponel$. 
For $i=1,2,3$, let $v_i$ be the direction of $z_{j,i,t}$ (viewed as an element of $\ponel$) at $y_j = \mathbf{M}_{j} (x_g)$ (where $ \mathbf{M}_{j} : \pberl \to \pberl$ is the 
rational map associated to $\{ M_{j,t} \}$).
Since $z=0, 1$ and $\infty$ are not holes of $H_j$, we conclude that $v_1, v_2, v_3$ are periodic good directions at $y_j$ under iterations of 
$\mishl{f}^{q_j}$. Hence $v_i$ is also a good direction at $y_j$ for $\mishl{f}^{\ell}$.
Moreover, 
$T_{y_j} \mishl{f}^\ell$ maps the three distinct
directions $v_1, v_2, v_3$ at $y_j$ onto three distinct directions
at $\mishl{f}^{\ell}(y_j)$, since these directions are periodic. It follows that the direction 
$T_{y_j} \mishl{f}^\ell (v_i)$ contains $ f^{\ell}_t (z_{j,i,t})$.
Therefore, $$\mishl{f}^{\ell} (y_j) = \mathbf{\gamma}_{f^{\ell}_t (\vecz_{j,t})}(x_g).$$
That is,
$$\{f^\ell_t\} ([\{M_{j,t}\}]) = [\gamma_{f^{\ell}_t (\vecz_{j,t})}].$$
Thus, there exist non constant complex rational maps $\varphi_1, \varphi_2$ such that:
\begin{eqnarray}
  M_{j,t}^{-1} \circ f_t^\ell \circ \gamma_{f^\ell_t (\vecz_{j,t})} \to \vphi_1 ,\\
  \gamma_{f^\ell_t (\vecz_{j,t})}^{-1} \circ f_t^{q_j-\ell} \circ M_{j,t} \to \vphi_2, \\
M^{-1}_{k,t} \circ \gamma_{f^\ell_t (\vecz_{j,t})} \to M.
\end{eqnarray}

By continuity of $\hat{I}$, we have that
\begin{eqnarray}
  M_{j,n}^{-1} \circ f_n^\ell \circ \gamma_{f^\ell_n (\vecz_{j,n})} \to \vphi_1, \\
  \gamma_{f^\ell_n (\vecz_{j,n})}^{-1} \circ f_n^{q_j-\ell} \circ M_{j,n} \to \vphi_2, \\
M^{-1}_{k,n} \circ \gamma_{f^\ell_t (\vecz_{j,n})} \to M.
\end{eqnarray}

Thus, $ \{ M_{k,n} \} $ is equivalent to $\{ \gamma_{f^\ell_t (\vecz_{j,n})}\}$ and this latter sequence of frames  is dynamically dependent with  $\{ M_{j,n} \}$ for
$\seqfn$. By Lemma~\ref{lem:2}, we conclude that $ \{ M_{k,n} \} $  and $\{ M_{j,n} \}$ are dynamically dependent for $\seqfn$.
\end{proof}

\begin{corollary}
\label{cor:4}
  Assume that $\{ f_n \} \subset \ratdc$ and $f_n \to f \in \ratdc$.
Then every rescaling limit of  $\{ f_n \}$ is, modulo conjugacy,  an iterate of $f$.
\end{corollary}

\begin{proof}
  Assume that $\seqMn$ is a rescaling for $\seqfn$ with limit $g$.
From the proposition, there exists a holomorphic family $\mish{f}$ with $f_0=f$ and
a moving frame $\mish{M}$ which is a rescaling for $\mish{f}$ with limit $g$. 
The associated rational map $\mishl{f}$ has reduction $f_0 =f$ of degree $d$.
Therefore, the Gauss point is completely invariant under $\mishl{f}: \pberl \to \pberl$.
In particular, the Julia set of $\mishl{f}$ is $\{ x_g \}$. It follows that $\mishl{M} (x_g) = x_g$, for otherwise there would be a repelling
periodic orbit of $\mishl{f}$ different from the Gauss point, according to Proposition~\ref{pro:2}. 
This proposition also implies that the rescaling limit of $\mish{f}$ associated to $\mish{M}$ is an iterate of $f$, modulo conjugacy.
\end{proof}

\begin{proof}[Proof of Theorem~\ref{ithr:2}]
Given any finite collection $\cC$ of pairwise dynamically independent rescalings for a sequence $\seqfn$ of rational maps, Proposition~\ref{pro:1} guarantees the existence of a holomorphic family $\mish{f}$ with a corresponding collection of pairwise dynamically independent moving frames which are rescalings for  $\mish{f}$ with the same period and rescaling limits as the ones of $\cC$. Hence,   Theorem~\ref{ithr:3} implies Theorem~\ref{ithr:2}.
\end{proof}


\end{document}